\documentclass[12pt,reqno]{amsart}
\usepackage{amssymb,amsthm}
\usepackage{braket,caption,dsfont,enumerate,esint,float,mathdots,mathtools}
\usepackage{perpage,pdfpages,pgfplots,thmtools}
\usepackage[margin=2.2cm]{geometry}
\usepackage{tikz,tikz-3dplot}

\usepackage[T1]{fontenc}
\usepackage[english]{babel}
\usepackage{lmodern,microtype,hyperref}

\allowdisplaybreaks
\pgfplotsset{compat=1.17}
\MakePerPage{footnote}

\DeclareMathOperator{\dist}{dist}

\DeclareMathOperator{\Rer}{Re}

\DeclareMathOperator{\Span}{span}

\DeclareMathOperator{\Vol}{Vol}

\newcommand{\C}{\mathbb{C}}

\newcommand{\Hc}{\mathcal{H}}

\newcommand{\HS}{\mathrm{HS}}

\newcommand{\Id}{\mathrm{I}}

\newcommand{\M}{\mathbb{M}}

\newcommand{\N}{\mathbb{N}}

\newcommand{\ol}{\overline}

\newcommand{\R}{\mathbb{R}}

\newcommand{\Sb}{\mathbb{S}}

\newcommand{\SH}{\mathbb{SH}}

\newcommand{\SO}{\mathrm{SO}}

\newcommand{\ve}{\varepsilon}

\theoremstyle{plain}
\newtheorem{thm}{Theorem}[section]

\newtheorem{lemma}[thm]{Lemma}

\newtheorem{prob}{Problem}

\theoremstyle{definition}
\newtheorem*{defn}{Definition}
\newtheorem*{ex}{Example}

\theoremstyle{remark}
\newtheorem*{rmk}{Remark}

\numberwithin{equation}{section}

\title{Orthonormal bases of spherical harmonics with maximal $L^p$ norm growth}
\author{Xiaolong Han}
\email{xiaolong.han@csun.edu}
\address{Department of Mathematics, California State University, Northridge, CA 91330, USA}

\subjclass[2020]{Primary 33C55, 58J50; Secondary 35P20, 05B30}

\keywords{Eigenfunction estimates, spherical harmonics, maximal $L^p$ norm growth, spherical designs, restricted invertibility}

\begin{document}
\begin{abstract}
In this paper, we show that for every $p>2$, there is an orthonormal basis of spherical harmonics on the $n$-dimensional unit sphere which achieves the maximal $L^p$ norm growth. More precisely, one basis achieves this growth for $p$ up to the critical exponent $\frac{2(n+1)}{n-1}$, and another for $p$ from the critical exponent up to $\infty$. The first basis is built from rotations of the real part of a highest weight spherical harmonic. The second is built from zonal harmonics whose poles are selected from a spherical design. In both cases, the selection uses the restricted invertibility theorem of Bourgain and Tzafriri. This extends earlier results of Bourgain and of the author on the 2-dimensional sphere for $2<p\le6$ to all dimensions and all $p>2$, and improves the positive density subsequence there to a whole basis.
\end{abstract}

\maketitle

\section{Introduction}
Let $\Delta_\M$ be the non-negative Laplacian on an $n$-dimensional compact Riemannian manifold $(\M,g)$ without boundary, where $n\ge2$. Let $u$ be a Laplacian eigenfunction such that $\Delta_\M u=\lambda^2u$ with $\lambda\ge1$. The $L^p$ estimates of eigenfunctions by H\"ormander \cite{Ho} ($p=\infty$) and Sogge \cite{So1, So2} ($2<p<\infty$) state that
\begin{equation}\label{eq:Sogge}
\|u\|_{L^p(\M)}\le C\lambda^{\sigma(p)}\|u\|_{L^2(\M)},
\end{equation}
where $C=C(\M,p)>0$ and
\[\sigma(p)=\begin{cases}
\frac{n-1}{4}-\frac{n-1}{2p}, & \text{if }2<p\le\frac{2(n+1)}{n-1}=p_n,\\
\frac{n-1}{2}-\frac np, & \text{if }p_n\le p\le\infty.
\end{cases}\]
Moreover, these estimates are sharp in terms of the growth rate as $\lambda\to\infty$ on the unit sphere $\Sb^n$ equipped with the round metric. Recall that Laplacian eigenfunctions on $\Sb^n$ are spherical harmonics, that is, harmonic and homogeneous polynomials in $\R^{n+1}$ restricted to $\Sb^n$; see, for example, Stein--Weiss \cite[Chapter IV]{SW}. Throughout the paper, we use $u$ to denote both the spherical harmonic on $\Sb^n$ and its homogeneous extension to $\R^{n+1}$ when no confusion can arise.

Denote by $\SH_k^n$ the space of spherical harmonics of degree $k$ on $\Sb^n$. Then for each $u\in\SH_k^n$,
\[\Delta_{\Sb^n}u=k(k+n-1)u,\]
which means that $\lambda=\sqrt{k(k+n-1)}\approx k$ for $k\ge1$. We also record that
\begin{equation}\label{eq:dim}
N_k:=\dim\SH_k^n=\frac{2k^{n-1}}{(n-1)!}+O_n\left(k^{n-2}\right).
\end{equation}
See Sogge \cite[Section 3.4]{So4}. The maximal $L^p$ norm growth in \eqref{eq:Sogge} is saturated by the highest weight spherical harmonics for $2<p\le p_n$ and by the zonal harmonics for $p_n\le p\le\infty$:

\begin{ex}[Highest weight spherical harmonics]
Let $x=(x_1,\dots,x_{n+1})\in\R^{n+1}$. Define
\begin{equation}\label{eq:q}
q(x)=(x_1+ix_2)^k.
\end{equation}
The restriction of $q$ to $\Sb^n$, still denoted by $q$, is in $\SH_k^n$, and is called a highest weight spherical harmonic. Following \cite{B2, Ha1}, we also call it a Gaussian beam. It is known that
\[\frac{\|q\|_{L^p(\Sb^n)}}{\|q\|_{L^2(\Sb^n)}}\approx k^{\frac{n-1}{4}-\frac{n-1}{2p}}.\]
See Sogge \cite[Theorem 4.1]{So1}. That is, highest weight spherical harmonics saturate the $L^p$ estimates \eqref{eq:Sogge} for all $2<p\le p_n$.
\end{ex}

\begin{ex}[Zonal harmonics]
Let $\{Y_j\}_{j=1}^{N_k}$ be an orthonormal basis of $\SH_k^n$. Consider the reproducing kernel in $\SH_k^n$:
\begin{equation}\label{eq:kernel}
K(x,y)=\sum_{j=1}^{N_k}Y_j(x)\ol{Y_j(y)}.
\end{equation}
The zonal harmonic with pole $x$ is defined as
\[Z_x(y)=\frac{K(y,x)}{\sqrt{K(x,x)}}\quad\text{for }y\in\Sb^n.\]
It is known that, for $p_n\le p\le\infty$,
\[\frac{\left\|Z_x\right\|_{L^p(\Sb^n)}}{\left\|Z_x\right\|_{L^2(\Sb^n)}}\approx k^{\frac{n-1}{2}-\frac np}.\]
See Sogge \cite[Theorem 4.1]{So1}. That is, zonal harmonics saturate the $L^p$ estimates \eqref{eq:Sogge} for all $p_n\le p\le\infty$.
\end{ex}

On $\Sb^2$, there are two Gaussian beams and one zonal harmonic in the standard basis of $\SH_k^2$. In view of $\dim\SH_k^2=2k+1$, these spherical harmonics have density zero in the standard basis of spherical harmonics in $L^2(\Sb^2)$. However, there is a great variety of eigenbases, and Sogge--Zelditch \cite{SZ3} asked whether there is a positive density subsequence of spherical harmonics (in some basis) which saturates the maximal $L^p$ norm growth. For $2<p\le p_2=6$, this question was answered by Bourgain \cite{B2} and, independently and by a different method, by the author \cite{Ha1}. More precisely, the following was proved in \cite{Ha1}.
\begin{thm}\label{thm:MaxSH2}
For each $k\in\N$, there is an orthonormal set $\{u_j\}_{j=1}^M\subset\SH_k^2$ such that $M\ge\left\lfloor\frac{2k+1}{400}\right\rfloor$ and
\[\left\|u_j\right\|_{L^p(\Sb^2)}\approx k^{\frac{1}{4}-\frac{1}{2p}}\quad\text{for all }j=1,\dots,M\text{ and }2<p\le6.\]
\end{thm}
Hence, the spherical harmonics constructed in \cite{Ha1} have density at least $\frac{1}{400}$. The corresponding problem for $p>6$ was left open; see \cite[Problem 5]{Ha1}. In this paper, we answer this problem and its higher-dimensional analogue. We construct whole orthonormal bases of spherical harmonics that saturate the maximal $L^p$ norm growth on $\Sb^n$ for all $p\ge p_n$.
\begin{thm}\label{thm:MaxSHhigh}
For each $k\in\N$, there is an orthonormal basis $\{u_j\}_{j=1}^{N_k}$ of $\SH_k^n$ such that
\[\left\|u_j\right\|_{L^p(\Sb^n)}\approx k^{\frac{n-1}{2}-\frac np}\quad\text{for all }j=1,\dots,N_k\text{ and }p_n\le p\le\infty.\]
Here, the implicit constants depend only on $n$ and $p$.
\end{thm}
The basis in Theorem \ref{thm:MaxSHhigh} is built from zonal harmonics whose poles are selected from a spherical design; see Section \ref{sec:outline}. Furthermore, we extend Theorem \ref{thm:MaxSH2} to higher dimensions and replace the positive density subsequence by a whole basis.
\begin{thm}\label{thm:MaxSHlow}
For each $k\in\N$, there is an orthonormal basis $\{u_j\}_{j=1}^{N_k}$ of $\SH_k^n$ such that
\[\left\|u_j\right\|_{L^p(\Sb^n)}\approx k^{\frac{n-1}{4}-\frac{n-1}{2p}}\quad\text{for all }j=1,\dots,N_k\text{ and }2<p\le p_n.\]
Here, the implicit constants depend only on $n$ and $p$.
\end{thm}

Taking the union of the bases in Theorem \ref{thm:MaxSHhigh} over all $k\in\N$ gives an orthonormal eigenfunction basis in $L^2(\Sb^n)$ which saturates the maximal $L^p$ norm growth for all $p_n\le p\le\infty$. Theorem \ref{thm:MaxSHlow} gives such a basis for all $2<p\le p_n$. We do not know whether a single orthonormal basis of $\SH_k^n$ can saturate \eqref{eq:Sogge} for all $2<p\le\infty$ simultaneously. We propose the following problem.
\begin{prob}\label{prob:basis}
Determine whether, for each $k\in\N$, there is an orthonormal basis $\{u_j\}_{j=1}^{N_k}$ of $\SH_k^n$ such that
\[\left\|u_j\right\|_{L^p(\Sb^n)}\approx k^{\sigma(p)}\quad\text{for all }j=1,\dots,N_k\text{ and }2<p\le\infty,\]
where the implicit constants depend only on $n$ and $p$.
\end{prob}
See the remark at the end of Section \ref{sec:low} for why Theorems \ref{thm:MaxSHhigh} and \ref{thm:MaxSHlow} do not solve this problem.

\subsection{Outline of the proof}\label{sec:outline}
For $p\ge p_n$, we use zonal harmonics. If we choose $M$ normalized zonal harmonics $\{Z_{x_j}\}_{j=1}^M$, they have maximal $L^p$ norm growth, but they are not orthogonal to each other in $L^2(\Sb^n)$. As in \cite{Ha1}, instead of using the Gram--Schmidt process, we modify them simultaneously. Write $Z=(Z_{x_1},\dots,Z_{x_M})^\top$, and set
\[E=\left(\left\langle Z_{x_j},Z_{x_l}\right\rangle\right)_{j,l=1}^M\quad\text{and}\quad u=E^{-\frac12}Z.\]
If $\{Z_{x_j}\}_{j=1}^M$ is linearly independent, then $E$ is positive definite, and the resulting spherical harmonics $u=(u_1,\dots,u_M)^\top$ form an orthonormal set in $L^2(\Sb^n)$.

In \cite{Ha1}, where the building blocks were Gaussian beams, the matrices $E$ and $E^{-\frac12}$ were both shown to be strictly diagonally dominant if the density is small enough, and the $L^p$ norms of $u_j$ were estimated by the triangle inequality. In this paper, we proceed differently. For zonal harmonics, the deleted absolute row sums of $E$ cannot be controlled, as observed in \cite[Section 3]{Ha1}. Instead, it is enough to find a lower bound for the smallest eigenvalue $\mu_{\min}(E)$. Indeed, we shall show in \eqref{eq:peak} that
\[u_j(x_j)=\sqrt{N_k}\left(E^{\frac12}\right)_{jj}\ge\sqrt{N_k}\mu_{\min}(E)^{\frac12}.\]
It then follows that a lower bound for $\mu_{\min}(E)$ yields one for $u_j(x_j)$, and thus for $\|u_j\|_{L^\infty(\Sb^n)}$.

To pass from the $L^\infty$ estimate to the $L^p$ ones, we use that a spherical harmonic of degree $k$ oscillates at the scale $k^{-1}$. So a lower bound for $u_j$ at the single point $x_j$ propagates to a ball of radius $\approx k^{-1}$. For $p\ge p_n$, the $L^p$ norm of $u_j$ is bounded below by its $L^p$ norm on this ball, which gives the maximal growth.

To obtain an orthonormal set of $M\approx N_k$ spherical harmonics with maximal $L^p$ norm growth for $p\ge p_n$, it is therefore enough to find $M$ zonal harmonics for which $\mu_{\min}(E)$ is bounded from below uniformly in $k$. To this end, we place the poles at the points $x_1,\dots,x_{M_0}$ of a spherical $2k$-design. Then the zonal harmonics form a tight frame of $\SH_k^n$, that is,
\[\sum_{j=1}^{M_0}Z_{x_j}\otimes Z_{x_j}=\frac{M_0}{N_k}\Id,\]
where $Z_{x_j}\otimes Z_{x_j}$ is the rank one operator $v\mapsto\langle v,Z_{x_j}\rangle Z_{x_j}$; see Lemma \ref{lemma:frame}. This is the input of the restricted invertibility theorem of Bourgain--Tzafriri \cite{BT}, in the form of Theorem \ref{thm:RI}, which selects $M\approx N_k$ of these zonal harmonics with $\mu_{\min}(E)$ bounded from below uniformly in $k$. In this way, spherical designs, a classical notion of well-distributed point configurations, produce bases of spherical harmonics with maximal $L^p$ norm growth. See also \cite{Ha2} for other connections between point configurations on the sphere and spherical harmonics.

The same argument does not give the maximal growth for $2<p<p_n$. Indeed, a ball of radius $\approx k^{-1}$ only gives the exponent $\frac{n-1}{2}-\frac np$, which is strictly smaller than $\frac{n-1}{4}-\frac{n-1}{2p}$ when $p<p_n$. Thus, for the low $p$ range, we use different building blocks, namely, the normalized real parts of Gaussian beams. Let $Q$ be one such spherical harmonic, that is, the normalized real part of $q$ in \eqref{eq:q}. Denote by $Q_R$ its rotation by $R\in\SO(n+1)$. We shall show that
\[\int_{\SO(n+1)}Q_R\otimes Q_R\,dR=\frac{1}{N_k}\Id,\]
where $dR$ is the normalized Haar measure on $\SO(n+1)$ and $Q_R\otimes Q_R$ is the rank one operator $v\mapsto\langle v,Q_R\rangle Q_R$. We approximate this integral by a finite sum and apply Theorem \ref{thm:RI}. This gives $M\approx N_k$ rotations $Q_j=Q_{R_j}$, $j=1,\dots,M$, whose matrix $E=(\langle Q_j,Q_l\rangle)$ has smallest eigenvalue bounded from below. Let
\[F=E^{-\frac12}\quad\text{and}\quad u_j=\sum_{l=1}^MF_{jl}Q_l.\]
Then
\[\left\langle u_j,Q_j\right\rangle=\left(E^{\frac12}\right)_{jj}\ge\mu_{\min}(E)^{\frac12}.\]
For the low $p$ range, instead of estimating $u_j$ on a small ball, we use the $L^p$--$L^{p'}$ duality through H\"older's inequality. If $p'=\frac{p}{p-1}$, then
\[\left|\left\langle u_j,Q_j\right\rangle\right|\le\left\|u_j\right\|_{L^p(\Sb^n)}\left\|Q_j\right\|_{L^{p'}(\Sb^n)}.\]
Therefore, a lower bound
\[\left\|u_j\right\|_{L^p(\Sb^n)}\gtrsim k^{\sigma(p)}\quad\text{for }2<p\le p_n\]
follows from a lower bound for $\mu_{\min}(E)$ together with
\[\|Q_j\|_{L^{p'}(\Sb^n)}\approx k^{-\sigma(p)}.\]
This duality argument does not give the maximal growth for $p>p_n$: it only gives the exponent $\frac{n-1}{4}-\frac{n-1}{2p}$, which is strictly smaller than $\frac{n-1}{2}-\frac np$ when $p>p_n$.

To obtain a whole basis in either construction, it is enough to take $M\ge\frac{N_k}{2}$. Complete $\{u_j\}_{j=1}^M$ to an orthonormal basis and pair each of the additional vectors with a distinct $u_j$. By an orthogonal change of basis in each pair, we can arrange that both resulting vectors have inner product at least $\frac{c}{\sqrt2}$ with the corresponding $Z_{x_j}$ or $Q_j$, where $c>0$ is the lower bound for the original inner product. The pointwise estimate or the duality argument then applies to every vector in the resulting basis. We prove this elementary fact in Lemma \ref{lemma:completion}. Consequently, it suffices to apply Theorem \ref{thm:RI} with parameters fixed independently of $k$, which keeps all the implicit constants uniform in $k$.

\subsection{Manifolds with maximal \texorpdfstring{$L^p$}{Lp} norm growth}
The $L^p$ norm estimates of eigenfunctions \eqref{eq:Sogge} are sharp on $\Sb^n$, and Theorems \ref{thm:MaxSHhigh} and \ref{thm:MaxSHlow} give whole eigenbases with maximal norm growth. However, these estimates are rarely sharp on a general manifold, and it is natural to ask on which manifolds they are. Extending the terminology of \cite{SZ1}, where the case $p=\infty$ was considered, we say that $(\M,g)$ has maximal $L^p$ norm growth if there is an orthonormal eigenfunction basis $\{u_j\}$ in $L^2(\M)$, with $\Delta_\M u_j=\lambda_j^2u_j$, such that
\[\limsup_{j\to\infty}\frac{\left\|u_j\right\|_{L^p(\M)}}{\lambda_j^{\sigma(p)}}>0.\]
Below, the eigenfunctions are always normalized in $L^2(\M)$.

For $p>p_n$, Sogge--Zelditch \cite{SZ1} and Sogge--Toth--Zelditch \cite{STZ} found necessary conditions for maximal $L^p$ norm growth, which involve the points at which the geodesics refocus. In particular, \cite[Theorem 1.4]{SZ1} shows that, for a generic metric $g$ on $\M$, $\|u\|_{L^\infty(\M)}=o\left(\lambda^{\frac{n-1}{2}}\right)$ for all eigenfunctions $u$, so that $(\M,g)$ does not have maximal $L^\infty$ norm growth. More recently, Canzani--Galkowski \cite{CG} gave dynamical conditions which guarantee quantitative improvements over \eqref{eq:Sogge} for $p>p_n$. They also described the eigenfunctions that nearly saturate these estimates in this range as finite sums of quasimodes approximating zonal harmonics. These results are in accordance with the example of zonal harmonics above: on $\Sb^n$, all the geodesics emanating from a given point return to it at the same time, and the zonal harmonic with that point as its pole concentrates there. We refer to \cite{CG, STZ, SZ1} for details and other related results.

For $2<p<p_n$, fewer results are known. Sogge \cite{So3} on compact surfaces, and Blair--Sogge \cite{BS} in higher dimensions, proved that a sequence of eigenfunctions $u$ satisfies $\|u\|_{L^p(\M)}=o\left(\lambda^{\sigma(p)}\right)$ if and only if the $L^2$ mass of $u$ in the $\lambda^{-\frac12}$-neighborhood of every unit-length geodesic is $o(1)$, uniformly in the geodesic. Again, this is in accordance with the examples above: the highest weight spherical harmonics on $\Sb^n$ concentrate in such a neighborhood of a great circle. On a compact surface, Sogge--Zelditch \cite{SZ3} proved that if the set of periodic geodesics has measure zero, then every orthonormal eigenfunction basis contains a full density subsequence $\{u_{j_m}\}$ with $\|u_{j_m}\|_{L^p(\M)}=o(\lambda_{j_m}^{\sigma(p)})$ for $2<p<6$. Hence, on such a surface, no orthonormal eigenfunction basis contains a positive density subsequence whose $L^p$ norms saturate \eqref{eq:Sogge} for some $2<p<6$. In contrast, on $\Sb^2$, Theorem \ref{thm:MaxSHlow} provides an orthonormal eigenfunction basis all of whose non-constant elements saturate \eqref{eq:Sogge} for all $2<p\le6$. In this respect, the sphere is again extremal. This is the contrast which motivated the question of \cite{SZ3} recalled above.

\subsection{Spherical harmonics with minimal \texorpdfstring{$L^p$}{Lp} norm growth}
In the opposite direction to the questions investigated in this paper, one asks for the minimal $L^p$ growth of eigenfunctions, the strongest form of which is to find eigenfunctions that are uniformly bounded as the eigenvalues tend to infinity. Such examples exist on flat tori, namely the standard exponential functions. On $\Sb^3$ and on $\Sb^5$, regarded as the boundaries of the unit balls in $\C^2$ and in $\C^3$, Bourgain \cite{B1, B3} constructed uniformly bounded orthonormal bases of the spaces of homogeneous holomorphic polynomials, which are also spherical harmonics. See \cite{Ha3} for the description of the mass distribution of these uniformly bounded spherical harmonics on $\Sb^3$.

On compact K\"ahler manifolds, Shiffman \cite{Sh} constructed uniformly bounded orthonormal systems of holomorphic sections of the $k$-th power of a positive line bundle, whose cardinality is at least a fixed proportion of the dimension of the space of such sections. Marzo--Ortega-Cerd\`a \cite{MOC} improved this proportion to $1-\ve$ for every $\ve>0$. These results imply the existence of uniformly bounded spherical harmonics on odd-dimensional spheres. Very recently, Gon\c{c}alves, Ferreira, and Radziwi{\l\l} \cite{GFR} and Wang, Zhang, and Zhu \cite{WZZ} constructed uniformly bounded orthonormal bases of the spaces of homogeneous holomorphic polynomials on $\Sb^{2d-1}\subset\C^d$ for every $d\ge2$. The existence of uniformly bounded spherical harmonics is open on even-dimensional spheres; see \cite{Ha2} for this problem on $\Sb^2$ and its relation to well-distributed points on the sphere.

\subsection{Notation}
Throughout this paper, $A\lesssim B$ ($A\gtrsim B$) means $A\le cB$ ($A\ge cB$) for some constant $c>0$ depending only on $n$ and $p$; $A\approx B$ means $A\lesssim B$ and $B\lesssim A$; the implicit constants may vary from line to line.

\subsection{Organization of the paper}
In Section \ref{sec:pre}, we collect the facts used in the proofs: the $L^r$ norms of the highest weight spherical harmonics and their rotations, the reproducing kernel and the zonal harmonics, a gradient estimate, the restricted invertibility theorem, spherical designs, and a lemma on orthonormal bases. In Section \ref{sec:high}, we prove Theorem \ref{thm:MaxSHhigh}. In Section \ref{sec:low}, we prove Theorem \ref{thm:MaxSHlow}.

\section{Preliminaries}\label{sec:pre}
In this section, we collect the terminology and facts used in the proofs of Theorems \ref{thm:MaxSHhigh} and \ref{thm:MaxSHlow}. From now on, $dy$ denotes the Riemannian volume form on $\Sb^n$, normalized to be a probability measure, and all the $L^r(\Sb^n)$ norms and the volumes $\Vol$ of subsets of $\Sb^n$ are taken with respect to it.

We also regard $\SH_k^n$ as the real vector space of real-valued spherical harmonics of degree $k$, of dimension $N_k$, with the inner product
\[\langle u,v\rangle=\int_{\Sb^n}uv\,dy.\]
In particular, the basis $\{Y_j\}_{j=1}^{N_k}$ in \eqref{eq:kernel} is taken to be real-valued, so that $K$ is real and symmetric. A real-valued orthonormal basis of this real space is also an orthonormal basis of the complex space of spherical harmonics of degree $k$. Moreover, rescaling the measure by a constant, and the basis accordingly, changes the $L^r(\Sb^n)$ norms only by factors depending on $n$ and $r$. So it suffices to prove Theorems \ref{thm:MaxSHhigh} and \ref{thm:MaxSHlow} for real-valued bases, with all norms taken with respect to $dy$.

For $u\in\SH_k^n$, we write $u\otimes u$ for the rank one operator on $\SH_k^n$ given by $v\mapsto\langle v,u\rangle u$.

\subsection{Highest weight spherical harmonics}
In \cite{Ha1}, the Gaussian beams are complex-valued. Since Theorem \ref{thm:RI} is stated over $\R$, here we use their real parts. Recall the highest weight spherical harmonic $q(x)=(x_1+ix_2)^k$ in \eqref{eq:q}. Set
\[Q=\frac{\Rer q}{\left\|\Rer q\right\|_{L^2(\Sb^n)}}.\]
Then $Q\in\SH_k^n$ is real-valued and $\|Q\|_{L^2(\Sb^n)}=1$. We record the $L^r$ norms of $Q$.
\begin{lemma}\label{lemma:HWnorm}
Let $1\le r<\infty$. Then
\[\left\|Q\right\|_{L^r(\Sb^n)}\approx k^{\frac{n-1}{4}-\frac{n-1}{2r}}.\]
Here, the implicit constants depend on $n$ and $r$.
\end{lemma}
\begin{proof}
Write $x\in\Sb^n$ as
\[x=\left(\rho\cos\theta,\rho\sin\theta,\sqrt{1-\rho^2}\omega\right)\quad\text{for }0\le\rho\le1\text{, }0\le\theta<2\pi\text{, and }\omega\in\Sb^{n-2}.\]
Up to a constant depending only on $n$, the volume form on $\Sb^n$ is
\[\rho\left(1-\rho^2\right)^{\frac{n-3}{2}}\,d\rho d\theta d\omega,\]
where $d\omega$ is the volume form on $\Sb^{n-2}$ (the counting measure on $\Sb^0=\{\pm1\}$ if $n=2$). Since $\Rer q(x)=\rho^k\cos(k\theta)$, and
\[\int_0^{2\pi}\left|\cos(k\theta)\right|^r\,d\theta\text{ is independent of $k$},\]
we have
\[\left\|\Rer q\right\|_{L^r(\Sb^n)}^r\approx\int_0^1\rho^{kr+1}\left(1-\rho^2\right)^{\frac{n-3}{2}}\,d\rho=\frac12B\left(\frac{kr}{2}+1,\frac{n-1}{2}\right)\approx k^{-\frac{n-1}{2}},\]
where $B$ is the Beta function and we substituted $s=\rho^2$. In the last step, we used Stirling's formula, which gives
\[B(a,b)=\frac{\Gamma(a)\Gamma(b)}{\Gamma(a+b)}\approx a^{-b}\quad\text{for }a\ge1,\]
with constants depending only on $b$. Hence, $\left\|\Rer q\right\|_{L^r(\Sb^n)}\approx k^{-\frac{n-1}{2r}}$ for all $1\le r<\infty$. In particular, $\left\|\Rer q\right\|_{L^2(\Sb^n)}\approx k^{-\frac{n-1}{4}}$, which proves the lemma.
\end{proof}

For $R\in\SO(n+1)$, set
\[Q_R(x)=Q\left(R^{-1}x\right).\]
Since $\Delta_{\Sb^n}$ and $dy$ are invariant under rotations, $Q_R\in\SH_k^n$, and $\|Q_R\|_{L^r(\Sb^n)}=\|Q\|_{L^r(\Sb^n)}$ for all $r$.

We shall also use the following fact. Let $dR$ be the normalized Haar measure on $\SO(n+1)$.
\begin{lemma}\label{lemma:HWframe}
We have
\[\int_{\SO(n+1)}Q_R\otimes Q_R\,dR=\frac{1}{N_k}\Id\]
as an operator on $\SH_k^n$. Consequently, for every $\eta>0$, there are $M_0\ge N_k$ rotations $R_1,\dots,R_{M_0}$, not necessarily distinct, such that
\begin{equation}\label{eq:HWframe}
\left\|\sum_{j=1}^{M_0}Q_{R_j}\otimes Q_{R_j}\right\|\le(1+\eta)\frac{M_0}{N_k}.
\end{equation}
Here, $\|\cdot\|$ is the operator norm on $\SH_k^n$.
\end{lemma}
\begin{proof}
For the first statement, see Jakobson--Zelditch \cite[Section 3]{JZ}. We provide a short proof for completeness. Denote the operator on the left-hand side of the first equality by $T$. For $S\in\SO(n+1)$, the map $U_S:u\mapsto u\left(S^{-1}\cdot\right)$ is orthogonal on $\SH_k^n$, and $U_SQ_R=Q_{SR}$. Hence,
\[U_S\left(Q_R\otimes Q_R\right)U_S^{-1}=Q_{SR}\otimes Q_{SR},\]
and the invariance of $dR$ gives $U_STU_S^{-1}=T$. That is, the symmetric operator $T$ commutes with all rotations, so each eigenspace of $T$ is invariant under rotations. The real vector space $\SH_k^n$ is irreducible under the action of $\SO(n+1)$, because its complexification, the space of complex-valued spherical harmonics of degree $k$, is irreducible; see, for example, Stein--Weiss \cite[Chapter IV]{SW}. Therefore, $T=c\Id$ for some $c\in\R$. Taking the trace and using $\|Q_R\|_{L^2(\Sb^n)}=1$, we obtain $c=N_k^{-1}$.

For the second statement, since $\SO(n+1)$ is compact and $R\mapsto Q_R\otimes Q_R$ is continuous, we approximate the integral in operator norm by a finite sum
\[\sum_ja_jQ_{R_j}\otimes Q_{R_j},\quad\text{where }a_j>0\text{ and }\sum_ja_j=1,\]
such that
\[\left\|\sum_ja_jQ_{R_j}\otimes Q_{R_j}-\frac{1}{N_k}\Id\right\|<\frac{\eta}{2N_k}.\]
For instance, this can be done by partitioning $\SO(n+1)$ into finitely many Borel sets of small diameter and positive measure.

We then approximate the $a_j$ by positive rational numbers $a_j'$ such that
\[\sum_ja_j'=1\quad\text{and}\quad\sum_j\left|a_j-a_j'\right|<\frac{\eta}{2N_k}.\]
Since $\|Q_R\otimes Q_R\|=1$, replacing $a_j$ by $a_j'$ changes the finite sum by less than $\frac{\eta}{2N_k}$ in operator norm.

Write
\[a_j'=\frac{m_j}{M_0},\quad\text{where }m_j\in\N\text{ and }M_0=\sum_jm_j.\]
Repeating each rotation $R_j$ exactly $m_j$ times, and relabeling, we obtain
\[\left\|\frac{1}{M_0}\sum_{j=1}^{M_0}Q_{R_j}\otimes Q_{R_j}-\frac{1}{N_k}\Id\right\|<\frac{\eta}{N_k}.\]
This gives \eqref{eq:HWframe}. By repeating each rotation the same number of times if necessary, we may assume that $M_0\ge N_k$.
\end{proof}

\subsection{Zonal harmonics}
Recall the reproducing kernel $K$ in \eqref{eq:kernel}, where $\{Y_j\}_{j=1}^{N_k}$ is now a real-valued orthonormal basis of $\SH_k^n$. Then $K$ does not depend on the choice of the basis and
\[u(x)=\int_{\Sb^n}K(x,y)u(y)\,dy\quad\text{for all }u\in\SH_k^n\text{ and }x\in\Sb^n.\]
For $R\in\SO(n+1)$, since $\Delta_{\Sb^n}$ and $dy$ are invariant under rotations, $\{Y_j(R^{-1}\cdot)\}_{j=1}^{N_k}$ is also an orthonormal basis of $\SH_k^n$, so
\[K\left(R^{-1}x,R^{-1}y\right)=K(x,y).\]
Since $\SO(n+1)$ acts transitively on $\Sb^n$, the function $x\mapsto K(x,x)$ is constant. Hence,
\[K(x,x)=\int_{\Sb^n}K(y,y)\,dy=\sum_{j=1}^{N_k}\left\|Y_j\right\|_{L^2(\Sb^n)}^2=N_k\quad\text{for all }x\in\Sb^n.\]
Therefore, the zonal harmonic
\[Z_x(y)=\frac{K(y,x)}{\sqrt{K(x,x)}}=\frac{K(y,x)}{\sqrt{N_k}}\]
satisfies $\|Z_x\|_{L^2(\Sb^n)}=1$ and
\begin{equation}\label{eq:eval}
\left\langle u,Z_x\right\rangle=\frac{u(x)}{\sqrt{N_k}}\quad\text{for all }u\in\SH_k^n.
\end{equation}
Since $\|Z_x\|_{L^2(\Sb^n)}=1$, the Cauchy--Schwarz inequality and \eqref{eq:eval} give
\begin{equation}\label{eq:sup}
\|u\|_{L^\infty(\Sb^n)}\le\sqrt{N_k}\|u\|_{L^2(\Sb^n)}\quad\text{for }u\in\SH_k^n.
\end{equation}
This is H\"ormander's estimate \eqref{eq:Sogge} on $\Sb^n$ when $p=\infty$. Moreover, taking $u=Z_y$ in \eqref{eq:eval},
\begin{equation}\label{eq:ZZ}
Z_y(x)=\sqrt{N_k}\left\langle Z_y,Z_x\right\rangle.
\end{equation}
In particular, $Z_x(x)=\sqrt{N_k}$, so each zonal harmonic $Z_x$ saturates \eqref{eq:sup} at its pole $x$.

\subsection{A gradient estimate}
The following estimate on the gradient of $u\in\SH_k^n$ allows us to pass from the value of $u$ at a point to its values at neighboring points at distance $\lesssim k^{-1}$. On general compact manifolds, the corresponding estimate follows from the local Weyl law; see Sogge--Zelditch \cite[Section 1]{SZ2}. On $\Sb^n$, we include a short proof based on the rotational invariance used above.

\begin{lemma}\label{lemma:grad}
Let $u\in\SH_k^n$. Recall that $\lambda=\sqrt{k(k+n-1)}$. Then
\[\left\|\nabla u\right\|_{L^\infty(\Sb^n)}\le\lambda\sqrt{N_k}\|u\|_{L^2(\Sb^n)}.\]
\end{lemma}
\begin{proof}
Let $\{Y_j\}_{j=1}^{N_k}$ be a real-valued orthonormal basis of $\SH_k^n$. Set
\[G(x)=\sum_{j=1}^{N_k}\left|\nabla Y_j(x)\right|^2\quad\text{for }x\in\Sb^n.\]
Then $G$ does not depend on the choice of the basis. For $R\in\SO(n+1)$, applying this to the basis $\{Y_j(R^{-1}\cdot)\}_{j=1}^{N_k}$ and using that $R$ is an isometry of $\Sb^n$, we obtain $G(R^{-1}x)=G(x)$. Hence, $G$ is constant, and
\[G(x)=\sum_{j=1}^{N_k}\int_{\Sb^n}\left|\nabla Y_j(y)\right|^2\,dy=\sum_{j=1}^{N_k}\left\langle\Delta_{\Sb^n}Y_j,Y_j\right\rangle=\lambda^2N_k\quad\text{for all }x\in\Sb^n.\]
Now write
\[u=\sum_{j=1}^{N_k}a_jY_j,\quad\text{so that }\sum_{j=1}^{N_k}a_j^2=\|u\|_{L^2(\Sb^n)}^2.\]
For $x\in\Sb^n$ and a unit tangent vector $\xi\in T_x\Sb^n$, the Cauchy--Schwarz inequality gives
\[\left|\nabla u(x)\cdot\xi\right|=\left|\sum_{j=1}^{N_k}a_j\nabla Y_j(x)\cdot\xi\right|\le\|u\|_{L^2(\Sb^n)}\sqrt{G(x)}=\lambda\sqrt{N_k}\|u\|_{L^2(\Sb^n)}.\]
Taking the supremum over $\xi$ and $x$ proves the lemma.
\end{proof}

\subsection{Restricted invertibility}
The following theorem of Bourgain--Tzafriri \cite{BT} provides a large well-conditioned subset of a given set of vectors. We state it in the form of Spielman--Srivastava \cite[Theorem 2]{SS}. Here, $\|L\|$ and $\|L\|_{\HS}$ are the operator norm and the Hilbert--Schmidt norm of $L$, respectively.

\begin{thm}[Restricted invertibility]\label{thm:RI}
Let $0<\delta<1$ and $L:\R^d\to\R^d$ be a linear map. Suppose that $v_1,\dots,v_m\in\R^d$ satisfy
\[\sum_{j=1}^mv_jv_j^\top=\Id,\quad\text{where $\Id$ is the identity matrix}.\]
Then there is a subset $J\subset\{1,\dots,m\}$ with
\[|J|\ge\left\lfloor\frac{\delta^2\|L\|_{\HS}^2}{\|L\|^2}\right\rfloor\]
such that $\{Lv_j\}_{j\in J}$ is linearly independent and
\[\mu_{\min}\left(\sum_{j\in J}\left(Lv_j\right)\left(Lv_j\right)^\top\right)>\frac{(1-\delta)^2\|L\|_{\HS}^2}{m},\]
where $\mu_{\min}$ is computed on $\Span\{Lv_j\}_{j\in J}$.
\end{thm}

\subsection{Spherical designs}
We shall apply Theorem \ref{thm:RI} to a family $\{Z_{x_j}\}_{j=1}^M$ of zonal harmonics in $\SH_k^n$, where $\{x_j\}_{j=1}^M$ is a spherical design.

\begin{defn}[Spherical designs]
Let $t\in\N$. A finite subset $\{x_1,\dots,x_M\}\subset\Sb^n$ is called a spherical $t$-design if
\[\frac1M\sum_{j=1}^MP\left(x_j\right)=\int_{\Sb^n}P(y)\,dy\]
for every polynomial $P$ in $\R^{n+1}$ of degree at most $t$, restricted to $\Sb^n$.
\end{defn}

Bondarenko--Radchenko--Viazovska \cite{BRV} proved the following.
\begin{thm}\label{thm:BRV}
There is a constant $c_n>0$ such that for every $t\in\N$ and every $M\ge c_nt^n$, there is a spherical $t$-design in $\Sb^n$ consisting of $M$ points.
\end{thm}

\begin{lemma}\label{lemma:frame}
Let $\{x_1,\dots,x_M\}\subset\Sb^n$ be a spherical $2k$-design. Then
\[\sum_{j=1}^MZ_{x_j}\otimes Z_{x_j}=\frac{M}{N_k}\Id\]
as an operator on $\SH_k^n$.
\end{lemma}

\begin{proof}
Let $u\in\SH_k^n$. Then $u^2$ is the restriction to $\Sb^n$ of a polynomial of degree $2k$ in $\R^{n+1}$. Hence, by \eqref{eq:eval} and the spherical design property,
\[\left\langle\left(\sum_{j=1}^MZ_{x_j}\otimes Z_{x_j}\right)u,u\right\rangle=\sum_{j=1}^M\left\langle u,Z_{x_j}\right\rangle^2=\frac{1}{N_k}\sum_{j=1}^Mu(x_j)^2=\frac{M}{N_k}\int_{\Sb^n}u(y)^2\,dy=\frac{M}{N_k}\|u\|_{L^2(\Sb^n)}^2.\]
Since the operator on the left-hand side is symmetric, the lemma follows.
\end{proof}

\begin{rmk}
Lemma \ref{lemma:frame} says that the zonal harmonics with poles at a spherical $2k$-design form a tight frame of $\SH_k^n$. It is an exact discrete form of a continuous identity. Indeed, for $R\in\SO(n+1)$, the rotation invariance of $K$ gives $Z_x(R^{-1}y)=Z_{Rx}(y)$. So, for a fixed $x_0\in\Sb^n$, averaging over $\SO(n+1)$ as in Lemma \ref{lemma:HWframe} and using \eqref{eq:eval} give
\[\int_{\SO(n+1)}Z_{Rx_0}\otimes Z_{Rx_0}\,dR=\int_{\Sb^n}Z_y\otimes Z_y\,dy=\frac{1}{N_k}\Id.\]
A spherical $2k$-design discretizes the last integral exactly, without the approximation used for Gaussian beams in Lemma \ref{lemma:HWframe}. By Theorem \ref{thm:BRV}, such designs exist with $M\approx k^n$ points, which is the optimal order; see \cite{BRV}. Since $N_k\approx k^{n-1}$, for such a design, Theorem \ref{thm:RI} selects in Section \ref{sec:high} a proportion of order $k^{-1}$ of the points.
\end{rmk}

\subsection{An orthonormal basis}
The following elementary lemma allows us to pass from an orthonormal set of at least $\frac N2$ vectors in a space of dimension $N$ to a whole orthonormal basis, while preserving a uniform lower bound for the inner products used above.

\begin{lemma}\label{lemma:completion}
Let $\Hc$ be a real Hilbert space of dimension $N$. Suppose that $\{u_j\}_{j=1}^M$ is an orthonormal set in $\Hc$ with $M\ge\frac N2$. Assume that $w_1,\dots,w_M\in\Hc$ satisfy
\[\left\langle u_j,w_j\right\rangle\ge c>0\quad\text{for all }j=1,\dots,M.\]
Then there is an orthonormal basis $\{v_j\}_{j=1}^N$ of $\Hc$ such that, for each $j=1,\dots,N$, there is an index $l(j)\in\{1,\dots,M\}$ with
\[\left\langle v_j,w_{l(j)}\right\rangle\ge\frac{c}{\sqrt2}.\]
\end{lemma}

\begin{proof}
Complete $\{u_j\}_{j=1}^M$ to an orthonormal basis of $\Hc$ by adding $h_1,\dots,h_{N-M}$. Since $N-M\le M$, we can pair $h_j$ with $u_j$ for $j=1,\dots,N-M$. For each pair, set
\[a_j=\left\langle u_j,w_j\right\rangle,\quad b_j=\left\langle h_j,w_j\right\rangle,\quad\text{and}\quad r_j=\sqrt{a_j^2+b_j^2}\ge a_j\ge c.\]
Define
\[v_j^\pm=\frac{\left(a_j\mp b_j\right)u_j+\left(b_j\pm a_j\right)h_j}{\sqrt2r_j}.\]
Then $v_j^+$ and $v_j^-$ are orthonormal, span the same subspace as $u_j$ and $h_j$, and satisfy
\[\left\langle v_j^\pm,w_j\right\rangle=\frac{a_j(a_j\mp b_j)+b_j(b_j\pm a_j)}{\sqrt2r_j}=\frac{r_j}{\sqrt2}\ge\frac{c}{\sqrt2}.\]
Replacing each pair $u_j,h_j$ by $v_j^+,v_j^-$ and retaining the unpaired $u_j$ gives the required orthonormal basis.
\end{proof}

\section{Proof of Theorem \ref{thm:MaxSHhigh}}\label{sec:high}
In this section, we prove Theorem \ref{thm:MaxSHhigh}. Fix $k\in\N$. Set
\[\delta=\frac{\sqrt3}{2}\quad\text{and}\quad c_0=1-\delta>0.\]

\subsection{Selection of the poles}
By Theorem \ref{thm:BRV}, there is a spherical $2k$-design $\{x_1,\dots,x_{M_0}\}\subset\Sb^n$. By Lemma \ref{lemma:frame}, the zonal harmonics $Z_{x_1},\dots,Z_{x_{M_0}}$ span $\SH_k^n$, so $N_k\le M_0$. Hence, we may fix a linear isometry of $\SH_k^n$ onto a subspace of $\R^{M_0}$ and regard $\SH_k^n\subset\R^{M_0}$. We write $\langle\cdot,\cdot\rangle$ and $\|\cdot\|$ for the Euclidean inner product and norm on $\R^{M_0}$ as well; on $\SH_k^n$ they agree with the inner product and norm of $L^2(\Sb^n)$.

Let $\{e_j\}_{j=1}^{M_0}$ be the standard basis of $\R^{M_0}$. Then
\[\sum_{j=1}^{M_0}e_je_j^\top=\Id.\]
Let $L:\R^{M_0}\to\R^{M_0}$ be the linear map determined by
\[Le_j=Z_{x_j}\quad\text{for }j=1,\dots,M_0.\]
Since $\|Z_{x_j}\|_{L^2(\Sb^n)}=1$,
\[\|L\|_{\HS}^2=\sum_{j=1}^{M_0}\left\|Le_j\right\|^2=\sum_{j=1}^{M_0}\left\|Z_{x_j}\right\|_{L^2(\Sb^n)}^2=M_0.\]
Furthermore, for any $v\in\R^{M_0}$,
\[LL^\top v=\sum_{j=1}^{M_0}\left\langle L^\top v,e_j\right\rangle Le_j=\sum_{j=1}^{M_0}\left\langle v,Z_{x_j}\right\rangle Z_{x_j}.\]
Hence, $LL^\top$ vanishes on the orthogonal complement of $\SH_k^n$ in $\R^{M_0}$, while its restriction to $\SH_k^n$ is
\[\sum_{j=1}^{M_0}Z_{x_j}\otimes Z_{x_j}.\]
By Lemma \ref{lemma:frame},
\[\|L\|^2=\left\|LL^\top\right\|=\frac{M_0}{N_k}.\]
Applying Theorem \ref{thm:RI} to $L$ and the vectors $e_1,\dots,e_{M_0}$, with $m=d=M_0$, we obtain a subset $J\subset\{1,\dots,M_0\}$ with
\[M:=|J|\ge\left\lfloor\frac{\delta^2\|L\|_{\HS}^2}{\|L\|^2}\right\rfloor=\left\lfloor\frac34N_k\right\rfloor\]
such that $\{Z_{x_j}\}_{j\in J}$ is linearly independent and
\begin{equation}\label{eq:RIout}
\mu_{\min}\left(\sum_{j\in J}Z_{x_j}\otimes Z_{x_j}\right)>\frac{(1-\delta)^2\|L\|_{\HS}^2}{M_0}=c_0^2
\end{equation}
on $\Span\{Z_{x_j}\}_{j\in J}$. After relabeling, we may assume that $J=\{1,\dots,M\}$.

\subsection{Orthonormalization}
Denote by $E$ the $M\times M$ real symmetric matrix
\[E=\left(\left\langle Z_{x_j},Z_{x_l}\right\rangle\right)_{j,l=1}^M.\]
Let $A:\R^M\to\SH_k^n$ be the linear map with $Ae_j=Z_{x_j}$ for $j=1,\dots,M$, where $\{e_j\}_{j=1}^M$ now denotes the standard basis of $\R^M$. Denote by $A^\top:\SH_k^n\to\R^M$ its adjoint. Then
\[A^\top A=E\quad\text{and}\quad AA^\top=\sum_{j=1}^MZ_{x_j}\otimes Z_{x_j},\]
and the non-zero eigenvalues of $A^\top A$ and of $AA^\top$ coincide with multiplicities. Since $\{Z_{x_j}\}_{j=1}^M$ is linearly independent, $A$ is injective, $E$ is positive definite, and the image of $A$ is $\Span\{Z_{x_j}\}_{j=1}^M$. Hence, all the $M$ eigenvalues of $E$ are eigenvalues of $AA^\top$ restricted to $\Span\{Z_{x_j}\}_{j=1}^M$. As a consequence of \eqref{eq:RIout},
\begin{equation}\label{eq:mumin}
\mu_{\min}(E)>c_0^2.
\end{equation}
Let $F=E^{-\frac12}$ be the unique positive definite symmetric square root of $E^{-1}$; see, for example, Horn--Johnson \cite[Theorem 7.2.6]{HJ}. Define
\[u_j=\sum_{l=1}^MF_{jl}Z_{x_l}\quad\text{for }j=1,\dots,M.\]
Then, as in \cite{Ha1},
\[\left\langle u_j,u_{j'}\right\rangle=\left\langle\sum_{l=1}^MF_{jl}Z_{x_l},\sum_{l'=1}^MF_{j'l'}Z_{x_{l'}}\right\rangle=\sum_{l,l'=1}^MF_{jl}F_{j'l'}E_{ll'}=\left(FEF\right)_{jj'}=\Id_{jj'}\quad\text{for }j,j'=1,\dots,M,\]
so $\{u_j\}_{j=1}^M\subset\SH_k^n$ is an orthonormal set.

\subsection{\texorpdfstring{$L^\infty$}{L∞} estimate of \texorpdfstring{$u_j$}{uj}}
By \eqref{eq:ZZ},
\[u_j\left(x_j\right)=\sum_{l=1}^MF_{jl}Z_{x_l}\left(x_j\right)=\sqrt{N_k}\sum_{l=1}^MF_{jl}E_{lj}=\sqrt{N_k}\left(FE\right)_{jj}=\sqrt{N_k}\left(E^{\frac12}\right)_{jj}.\]
The matrix
\[E^{\frac12}-\mu_{\min}(E)^{\frac12}\Id\]
is positive semidefinite, so its diagonal entries are non-negative. Hence, by \eqref{eq:mumin},
\begin{equation}\label{eq:peak}
u_j\left(x_j\right)=\sqrt{N_k}\left(E^{\frac12}\right)_{jj}\ge\sqrt{N_k}\mu_{\min}(E)^{\frac12}>c_0\sqrt{N_k}.
\end{equation}
Therefore, by \eqref{eq:dim},
\[\left\|u_j\right\|_{L^\infty(\Sb^n)}\ge u_j\left(x_j\right)>c_0\sqrt{N_k}\gtrsim k^{\frac{n-1}{2}}.\]

\subsection{\texorpdfstring{$L^p$}{Lp} estimate of \texorpdfstring{$u_j$}{uj}}
Since $\|u_j\|_{L^2(\Sb^n)}=1$, Lemma \ref{lemma:grad} gives $|\nabla u_j|\le\lambda\sqrt{N_k}$ on $\Sb^n$. Set
\[B_j=\left\{y\in\Sb^n:\dist\left(x_j,y\right)\le\frac{c_0}{2\lambda}\right\}.\]
Then, for $y\in B_j$, the mean value theorem along a minimizing geodesic from $x_j$ to $y$ and \eqref{eq:peak} give
\[\left|u_j(y)\right|\ge u_j\left(x_j\right)-\lambda\sqrt{N_k}\cdot\dist\left(x_j,y\right)\ge\frac{c_0}{2}\sqrt{N_k}.\]
Since $\sqrt n\le\lambda\le\sqrt nk$, the radius of $B_j$ is at most $1$ and at least $\frac{c_0}{2\sqrt nk}$. Hence,
\[\Vol\left(B_j\right)\gtrsim k^{-n}.\]
Therefore, for $p_n\le p<\infty$,
\[\left\|u_j\right\|_{L^p(\Sb^n)}^p\ge\int_{B_j}\left|u_j(y)\right|^p\,dy\ge\left(\frac{c_0}{2}\right)^pN_k^{\frac p2}\Vol\left(B_j\right)\gtrsim N_k^{\frac p2}k^{-n},\]
and, by \eqref{eq:dim},
\[\left\|u_j\right\|_{L^p(\Sb^n)}\gtrsim N_k^{\frac12}k^{-\frac np}\approx k^{\frac{n-1}{2}-\frac np}.\]
Combining the $L^\infty$ and $L^p$ estimates above with the upper bound in \eqref{eq:Sogge} for $p_n\le p\le\infty$, we conclude that
\[\left\|u_j\right\|_{L^p(\Sb^n)}\approx k^{\frac{n-1}{2}-\frac np}\quad\text{for all }j=1,\dots,M\text{ and }p_n\le p\le\infty.\]

\subsection{An orthonormal basis}
Since $N_k\ge N_1=n+1\ge3$, we have
\[M\ge\left\lfloor\frac34N_k\right\rfloor\ge\left\lceil\frac{N_k}{2}\right\rceil.\]
By \eqref{eq:eval},
\[\left\langle u_j,Z_{x_j}\right\rangle=\frac{1}{\sqrt{N_k}}u_j\left(x_j\right).\]
So \eqref{eq:peak} gives
\[\left\langle u_j,Z_{x_j}\right\rangle>c_0\quad\text{for all }j=1,\dots,M.\]
Apply Lemma \ref{lemma:completion} with $\Hc=\SH_k^n$, $w_j=Z_{x_j}$, and $c=c_0$. We obtain an orthonormal basis $\{v_j\}_{j=1}^{N_k}$ of $\SH_k^n$ such that, by \eqref{eq:eval},
\[v_j\left(x_{l(j)}\right)=\sqrt{N_k}\left\langle v_j,Z_{x_{l(j)}}\right\rangle\ge\frac{c_0}{\sqrt2}\sqrt{N_k}\quad\text{for all }j=1,\dots,N_k.\]
The gradient estimate and the ball argument above, with $c_0$ replaced by $\frac{c_0}{\sqrt2}$, now give
\[\left\|v_j\right\|_{L^p(\Sb^n)}\approx k^{\frac{n-1}{2}-\frac np}\quad\text{for all }j=1,\dots,N_k\text{ and }p_n\le p\le\infty.\]
The constants depend only on $n$ and $p$, since $c_0$ is fixed. Thus, $\{v_j\}_{j=1}^{N_k}$ is the required basis, and this finishes the proof of Theorem \ref{thm:MaxSHhigh}.

\section{Proof of Theorem \ref{thm:MaxSHlow}}\label{sec:low}
In this section, we prove Theorem \ref{thm:MaxSHlow}. Fix $k\in\N$. Set
\[\eta=\frac18\quad\text{and}\quad\delta=\sqrt{\frac34(1+\eta)}=\sqrt{\frac{27}{32}}\in(0,1).\]
By Lemma \ref{lemma:HWframe}, there are rotations $R_1,\dots,R_{M_0}$ with $M_0\ge N_k$ such that, writing $Q_j=Q_{R_j}$,
\begin{equation}\label{eq:HWop}
\left\|\sum_{j=1}^{M_0}Q_j\otimes Q_j\right\|\le(1+\eta)\frac{M_0}{N_k}.
\end{equation}
As in Section \ref{sec:high}, we regard $\SH_k^n$ as a subspace of $\R^{M_0}$, and we write $\langle\cdot,\cdot\rangle$ and $\|\cdot\|$ for the Euclidean inner product and norm on $\R^{M_0}$ as well. Let $\{e_j\}_{j=1}^{M_0}$ be the standard basis of $\R^{M_0}$. Let $L:\R^{M_0}\to\R^{M_0}$ be the linear map determined by
\[Le_j=Q_j\quad\text{for }j=1,\dots,M_0.\]
Since $\|Q_j\|_{L^2(\Sb^n)}=1$,
\[\|L\|_{\HS}^2=M_0.\]
As in Section \ref{sec:high}, $LL^\top$ vanishes on the orthogonal complement of $\SH_k^n$ in $\R^{M_0}$, while its restriction to $\SH_k^n$ is
\[\sum_{j=1}^{M_0}Q_j\otimes Q_j.\]
Hence, by \eqref{eq:HWop},
\[\|L\|^2=\left\|LL^\top\right\|=\left\|\sum_{j=1}^{M_0}Q_j\otimes Q_j\right\|\le(1+\eta)\frac{M_0}{N_k}.\]
Applying Theorem \ref{thm:RI} to $L$ and the vectors $e_1,\dots,e_{M_0}$, with $m=d=M_0$, we obtain a subset $J\subset\{1,\dots,M_0\}$ with
\[M:=|J|\ge\left\lfloor\frac{\delta^2\|L\|_{\HS}^2}{\|L\|^2}\right\rfloor\ge\left\lfloor\frac{\delta^2}{1+\eta}N_k\right\rfloor=\left\lfloor\frac34N_k\right\rfloor\]
such that $\{Q_j\}_{j\in J}$ is linearly independent and
\begin{equation}\label{eq:HWRI}
\mu_{\min}\left(\sum_{j\in J}Q_j\otimes Q_j\right)>(1-\delta)^2
\end{equation}
on $\Span\{Q_j\}_{j\in J}$. After relabeling, we may assume that $J=\{1,\dots,M\}$.

Denote
\[E=\left(\left\langle Q_j,Q_l\right\rangle\right)_{j,l=1}^M.\]
As in Section \ref{sec:high}, $E$ is positive definite and \eqref{eq:HWRI} implies that
\[\mu_{\min}(E)>(1-\delta)^2.\]
Let $F=E^{-\frac12}$. Define
\[u_j=\sum_{l=1}^MF_{jl}Q_l\quad\text{for }j=1,\dots,M.\]
Then $FEF=\Id$, so $\{u_j\}_{j=1}^M$ is an orthonormal set in $\SH_k^n$. Moreover, as in \eqref{eq:peak},
\begin{equation}\label{eq:HWpair}
\left\langle u_j,Q_j\right\rangle=\sum_{l=1}^MF_{jl}E_{lj}=\left(FE\right)_{jj}=\left(E^{\frac12}\right)_{jj}\ge\mu_{\min}(E)^{\frac12}>1-\delta.
\end{equation}
Let $2<p\le p_n$. Denote by $p'=\frac{p}{p-1}$ the dual exponent. By Lemma \ref{lemma:HWnorm},
\[\|Q_j\|_{L^{p'}(\Sb^n)}\approx k^{\frac{n-1}{4}-\frac{n-1}{2p'}}=k^{-\frac{n-1}{4}+\frac{n-1}{2p}}=k^{-\sigma(p)}.\]
Hence, by H\"older's inequality and \eqref{eq:HWpair},
\[1-\delta<\left|\left\langle u_j,Q_j\right\rangle\right|\le\|u_j\|_{L^p(\Sb^n)}\|Q_j\|_{L^{p'}(\Sb^n)}\lesssim k^{-\sigma(p)}\|u_j\|_{L^p(\Sb^n)}.\]
Therefore,
\[\|u_j\|_{L^p(\Sb^n)}\gtrsim k^{\sigma(p)}=k^{\frac{n-1}{4}-\frac{n-1}{2p}}.\]
Combining this estimate with \eqref{eq:Sogge}, we conclude that
\[\|u_j\|_{L^p(\Sb^n)}\approx k^{\frac{n-1}{4}-\frac{n-1}{2p}}\quad\text{for all }j=1,\dots,M\text{ and }2<p\le p_n.\]
As in Section \ref{sec:high}, $M\ge\lfloor\frac34N_k\rfloor\ge\lceil\frac{N_k}{2}\rceil$. Apply Lemma \ref{lemma:completion} with $\Hc=\SH_k^n$, $w_j=Q_j$, and $c=1-\delta$, using \eqref{eq:HWpair}. This gives an orthonormal basis $\{v_j\}_{j=1}^{N_k}$ of $\SH_k^n$ such that
\[\left\langle v_j,Q_{l(j)}\right\rangle\ge\frac{1-\delta}{\sqrt2}\quad\text{for all }j=1,\dots,N_k.\]
For every $2<p\le p_n$, H\"older's inequality and Lemma \ref{lemma:HWnorm} therefore give
\[\frac{1-\delta}{\sqrt2}\le\left|\left\langle v_j,Q_{l(j)}\right\rangle\right|\le\left\|v_j\right\|_{L^p(\Sb^n)}\left\|Q_{l(j)}\right\|_{L^{p'}(\Sb^n)}\lesssim k^{-\sigma(p)}\left\|v_j\right\|_{L^p(\Sb^n)}.\]
Combining this with \eqref{eq:Sogge}, we obtain
\[\left\|v_j\right\|_{L^p(\Sb^n)}\approx k^{\frac{n-1}{4}-\frac{n-1}{2p}}\quad\text{for all }j=1,\dots,N_k\text{ and }2<p\le p_n.\]
The constants depend only on $n$ and $p$, since $\delta$ is fixed. Thus, $\{v_j\}_{j=1}^{N_k}$ is the required basis, and this finishes the proof of Theorem \ref{thm:MaxSHlow}.

\begin{rmk}
The bases in Theorems \ref{thm:MaxSHhigh} and \ref{thm:MaxSHlow} do not solve Problem \ref{prob:basis}, because each of them is controlled by only one mechanism. Each element of the basis in Section \ref{sec:high} has a value $\gtrsim\sqrt{N_k}$ at some point, which only gives the exponent $\frac{n-1}{2}-\frac np<\sigma(p)$ for $2<p<p_n$. Each element of the basis in this section has an inner product $\gtrsim1$ with some $Q_R$, which only gives the exponent $\frac{n-1}{4}-\frac{n-1}{2p}<\sigma(p)$ for $p>p_n$.

A single spherical harmonic can have both properties. For example, let $e_3=(0,0,1,0,\dots,0)$. Since $Q(e_3)=0$, \eqref{eq:eval} gives $\langle Z_{e_3},Q\rangle=0$. So the spherical harmonic
\[\psi=\frac{1}{\sqrt2}\left(Z_{e_3}+Q\right)\]
satisfies
\[\|\psi\|_{L^2(\Sb^n)}=1,\quad\psi\left(e_3\right)=\sqrt{\frac{N_k}{2}},\quad\langle\psi,Q\rangle=\frac{1}{\sqrt2}.\]
So the arguments in Sections \ref{sec:high} and \ref{sec:low} show that $\psi$ saturates \eqref{eq:Sogge} for all $2<p\le\infty$. However,
\[\psi\left(R^{-1}x\right)=\frac{1}{\sqrt2}(Z_{Re_3}(x)+Q_R(x)),\]
and the proof of Lemma \ref{lemma:HWframe} applies verbatim to $\psi$. So the construction in this section with $\psi$ in place of $Q$ only gives an orthonormal basis $\{v_j\}_{j=1}^{N_k}$ of $\SH_k^n$ such that, for each $j$,
\[\left\langle v_j,Z_{Re_3}\right\rangle+\left\langle v_j,Q_R\right\rangle\ge1-\delta\quad\text{for some }R\in\SO(n+1).\]
This does not force both inner products to be large.
\end{rmk}


\section*{Use of artificial intelligence}
The main ideas of this paper were developed by the author with the help of ChatGPT (OpenAI), in particular, through its suggestion to use the restricted invertibility theorem of Bourgain and Tzafriri \cite{BT}. The author also used Claude (Anthropic) to check the manuscript for errors and suggest revisions to the exposition. The author verified all suggestions and is responsible for the content of the paper.


\end{document}